\documentclass[11pt,letterpaper]{amsart}

\usepackage{amsmath,amssymb,amsthm,amsfonts}        
\usepackage{hyperref}        
\usepackage{vmargin}
\setpapersize[portrait]{USletter}
\setmarginsrb{1.1in}{0.9in}{1.1in}{0.7in}{12pt}{12pt}{12pt}{36pt}

\newtheorem{theorem}{Theorem}
\newcommand{\Theorem}[1]{Theorem~\ref{thm:#1}}
\newcommand{\TheoremName}[1]{\label{thm:#1}}

\newtheorem{claim}[theorem]{Claim}
\newcommand{\ClaimName}[1]{\label{clm:#1}}
\newcommand{\Claim}[1]{Claim~\ref{clm:#1}}
\newcommand{\SectionName}[1]{\label{sec:#1}}

\newcommand{\EquationName}[1]{\label{eq:#1}}

\newcommand{\set}[1]{\left \{ #1 \right \}}                     
\newcommand{\setst}[2]{\left\{\; #1 \,:\, #2 \;\right\}}        
\renewcommand{\th}{\ifmmode{^{\textrm{th}}}\else{\textsuperscript{th}\ }\fi}
\newcommand{\norm}[1]{\left\lVert #1 \right\rVert}
\newcommand{\abs}[1]{\lvert #1 \rvert}
\newcommand{\card}[1]{\abs{#1}}
\newcommand{\floor}[1]{\left\lfloor #1 \right\rfloor}
\newcommand{\smallsum}[2]{{\textstyle \sum_{#1}^{#2}}}
\newcommand{\defeq}{\,:=\,}                                     
\newcommand{\Union}{\bigcup}
\newcommand{\comment}[1]{}

\newcommand{\prob}[1]{\operatorname{Pr}\left[\,#1\,\right]}               

\newcommand{\bC}{\mathbb{C}}
\newcommand{\bR}{\mathbb{R}}
\newcommand{\cC}{\mathcal{C}}
\newcommand{\cE}{\mathcal{E}}
\newcommand{\cF}{\mathcal{F}}
\newcommand{\cY}{\mathcal{Y}}

\newtheorem{conj}[theorem]{Conjecture}

\DeclareRobustCommand{\myfn}{\textsuperscript{\fnsymbol{footnote}}}
\title[A note on the discrepancy of matrices with bounded row and column sums]
      {A note on the discrepancy of matrices \\ with bounded row and column sums}
\author{Nicholas J. A. Harvey\myfn}

\begin{document}
\maketitle 
\renewcommand{\thefootnote}{\fnsymbol{footnote}}
\footnotetext[1]{
    Department of Computer Science, University of British Columbia.
    Email: \texttt{nickhar@cs.ubc.ca}.
    Supported by an NSERC Discovery Grant and a Sloan Foundation Fellowship.
}
\renewcommand{\thefootnote}{\arabic{footnote}}


\begin{abstract}
A folklore result uses the Lov\'asz local lemma to analyze the discrepancy of
hypergraphs with bounded degree and edge size.
We generalize this result to the context of 
real matrices with bounded row and column sums.
\end{abstract}

\section{Introduction}

In combinatorics, discrepancy theory is the study of red-blue colorings
of a hypergraph's vertices such that every hyperedge contains a roughly equal number of red and blue
vertices.
A classic survey on this topic is \cite{BeckSos}.

Many combinatorial discrepancy results have a more general form as a geometric statement about
discrepancy of real vectors \cite[\S 4]{BeckSos}.
Some examples include the Beck-Fiala theorem \cite{BeckFiala}
and Spencer's ``six standard deviations'' theorem \cite{Spencer}.
One exception is the following folklore result on the discrepancy of hypergraphs of bounded degree
and edge size \cite[pp.~693]{Srinivasan} \cite[Proposition 12]{BPRS}.

\begin{theorem}
\TheoremName{folklore}
Let $H$ be a hypergraph of maximum degree $\Delta$ and maximum edge size $R$.
Then there is a red-blue coloring of the vertices such that, for every edge $e$,
the numbers of red and blue vertices in $e$ differ by at most $2 \sqrt{R \ln(R \Delta)}$.
\end{theorem}

The proof is a short exercise using the Lov\'asz local lemma.

We show that this theorem also has a more general form as a geometric statement
about discrepancy of real vectors.
\Theorem{newthm} recovers \Theorem{folklore} (up to constants)
by letting $V_{i,j} \in \set{0,1}$ indicate whether vertex $j$ is contained in edge $i$.
Let $v^i$ denote the $i\th$ row of $V$ and $v_j$ denote the $j\th$ column of $V$.
As usual, let $[n] = \set{1,\ldots,n}$ and let $\norm{\cdot}_p$ denote the $\ell_p$-norm.

\begin{theorem}
\TheoremName{newthm}
Let $V$ be an $n \times m$ real matrix with $\abs{V_{i,j}} \leq 1$,
$\norm{ v^i }_1 \leq R$, and $\norm{v_j}_1 \leq \Delta$ for all $i \in [n], j \in [m]$.
Assume that $R \geq \max \set{ \Delta, 4 }$ and $\Delta \geq 2$.
There exists $y \in \set{-1,+1}^m$ with 
$\norm{ V y }_\infty \leq O(\sqrt{R \log(R \Delta)})$.
\end{theorem}

\section{The Proof}
\SectionName{LLL}

\Theorem{newthm} follows as an easy corollary of the next theorem,
by rescaling the vectors and separately considering the positive and negative coordinates.
Let $\lg x$ denote the base-2 logarithm of $x$.

\begin{theorem}
\TheoremName{weakbeckfiala}
Let $A$ be a non-negative real matrix of size $n \times m$,
and let $a_1,\ldots,a_m \in \bR^n_{\geq 0}$ denote its columns.
Assume that $\beta \leq \min \set{ \delta/2, 1/4 }$ and $\delta \leq 1$.
Suppose that
\begin{itemize}
\item $\norm{\sum_j a_j}_\infty \leq 1$,
\item $\norm{a_j}_\infty \leq \beta$ for every $j$, and
\item $\norm{a_j}_1 \leq \delta$ for every $j$.
\end{itemize}
Define $\alpha := \sqrt{\lg(\delta/\beta^2)} \geq \sqrt{2}$.
Then there exists a vector $y \in \set{-1,+1}^m$ such that
$$
    \norm{ A y }_\infty ~\leq~ 16 \alpha \sqrt{\beta}.
$$
\end{theorem}

\comment{
\begin{proofof}{\Theorem{newthm}}
First we rescale each $v_j$ and divide it into its non-positive and non-negative parts.
Define $v_j^+, v_j^- \in \bR^n$ by
\begin{align*}
(v_j^+)_i &~=~ \max \set{  (v_j)_i/R, 0 } \\
(v_j^-)_i &~=~ \max \set{ -(v_j)_i/R, 0 }
\end{align*}
Let $a_j \in \bR^{2n}$ be the vector obtained by concatenating $v_j^+$ and $v_j^-$,
and let $A$ be the matrix whose $j\th$ column is $a_j$.
Then $\norm{ a_j }_\infty \leq 1/R$, $\norm{ a_j }_1 = \norm{ v_j }_1 / R \leq \Delta/R$,
and $\norm{ A_i }_1 \leq 1$.
Thus we may set $\delta = \Delta/R$ and $\beta = 1/R$.
By \Theorem{weakbeckfiala}, we get a vector $y \in \set{-1,+1}^m$ with
$$
\norm{ Ay }_\infty
    ~=~ O( \alpha \sqrt{\beta} )
    ~=~ O( \sqrt{ \lg( R \Delta ) / R } ).
$$
Scaling up by $R$ and using the triangle inequality,
we obtain a discrepancy bound for $v_1,\ldots,v_m$ that is at most twice as large:
$ \norm{ Vy }_\infty = O( \sqrt{ R \lg( R \Delta )} ) $.
\end{proofof}
}

We now prove \Theorem{weakbeckfiala}.
Suppose we choose the vector $y \in \set{-1,+1}^m$ uniformly at random.
The discrepancy of row $i$ is the value $\abs{\sum_j A_{i,j} y_j}$.
Our goal is to bound $\norm{Ay}_\infty = \max_i \, \abs{\sum_j A_{i,j} y_j}$,
which is the maximum discrepancy of any row.

One annoyance in analyzing $\norm{ A y }_\infty$ is that
the entries of $A$ can have wildly differing magnitudes.
The natural approach is to stratify: 
to partition each row of $A$ into sets whose entries all have roughly the same magnitude.
Define $b := \floor{-\lg \beta} \geq 2$,
so that every entry of every $A$ is at most $2^{-b}$.
For $k \geq b$, let
$$
    S_{i,k} ~=~ \setst{ j }{ \floor{-\lg A_{i,j}} = k }
$$
be the locations of the entries in row $i$ that take values in $(2^{-(k+1)},2^{-k}]$.

To bound the discrepancy of row $i$, we will actually bound the discrepancy of each
set $S_{i,k}$ (i.e., $\abs{ \sum_{j \in S_{i,k}} A_{i,j} y_j }$).
By the triangle inequality,
the total discrepancy of row $i$ is at most the sum of the discrepancies of each $S_{i,k}$.

Define
\begin{equation}
\EquationName{epsdef}
\epsilon ~:=~ 8 \alpha \sqrt{\beta} ~>~ 8 \sqrt{\beta}.
\end{equation}
Let $\cE_{i,k}$ be the event that the discrepancy of $S_{i,k}$ exceeds
\begin{equation}
\EquationName{TDef}
    T_k ~:=~ \epsilon \sum_{j \in S_{i,k}} A_{i,j} + \alpha 2^{-k/2}.
\end{equation}
We can analyze the probability of $\cE_{i,k}$ by a Hoeffding bound:
if $\set{X_i}_{i \leq \ell}$ are independent random variables,
each $X_i \in [-1,+1]$,
and $X = X_1+\cdots+X_\ell$,
then $\prob{|X|>a} \leq 2 e^{-a^2/2\ell}$. 
Applying this bound to the discrepancy of $S_{i,k}$, we get that
\begin{align}
    \nonumber
    \prob{\cE_{i,k}}
        &~\leq~ 2\exp\big( - (T_k 2^k)^2 / 2 \card{S_{i,k}} \big) \\\nonumber
        &~<~ 2\exp\Bigg(
            - \frac{\epsilon^2}{2 \card{S_{i,k}} } \Big(2^k \smallsum{j \in S_{i,k}}{} A_{i,j}\Big)^2
            - \frac{2 \epsilon}{2 \card{S_{i,k}} } \alpha 2^{k/2}
            \Big(2^k \smallsum{j \in S_{i,k}}{} A_{i,j}\Big)
            \Bigg) \\
        &~\leq~ 2\exp\Big(
            - \frac{\epsilon^2}{8} \card{S_{i,k}} - \frac{\epsilon}{2} \alpha2^{k/2} 
            \Big)
        ~=:~ p_{i,k},
        \EquationName{pikdef},
\end{align}
where the last inequality uses
$\smallsum{j \in S_{i,k}}{} A_{i,j} \geq 2^{-(k+1)} \card{S_{i,k}}$.

\subsection{Discrepancy assuming no events occur}
Suppose that none of the events $\cE_{i,k}$ happen.
Then the total discrepancy of row $i$ is at most
\begin{align}
\nonumber
\sum_{k \geq b} T_k
    &~=~ \epsilon
        \sum_{k \geq b} \sum_{j \in S_{i,k}} A_{i,j} 
        + \alpha \sum_{k \geq b} 2^{-k/2} \\\nonumber
    &~\leq~ \epsilon + \alpha \sum_{k \geq b} 2^{-k/2}
        \qquad\text{(since we assume $\smallsum{j=1}{m} A_{i,j} \leq 1$)} \\\nonumber
    &~=~ \epsilon + \alpha \frac{2^{-b/2}}{1-2^{-1/2}} \\\nonumber
    &~\leq~ \epsilon + 4 \alpha \sqrt{2 \beta}
        \qquad\text{(since $2^{-b} \leq 2^{-(\lg(1/\beta)-1)} = 2\beta$)}
        \\\EquationName{nobad}
    &~\leq~ 16 \alpha \sqrt{\beta}.
\end{align}

\subsection{Avoiding the events}
We will use the local lemma to show that, with positive probability,
none of the events $\cE_{i,k}$ occur.
To do so, we must show that these events have limited dependence.
Consider $\cE_{i,k}$, which is the event
that the elements in row $i$ of value roughly $2^{-k}$ have large discrepancy.
This event depends only on the random values $\setst{ y_j }{ j \in S_{i,k} }$.
We will bound the total failure probability of the events that depend on those random values.

The local lemma can be stated as follows \cite[Theorem 5.1.1]{AlonSpencer}:

\begin{theorem}
Let $\cE_1,\ldots,\cE_m$ be events in a probability space.
Let $\Gamma(\cE_i)$ be the events (other than $\cE_i$ itself)
which are not independent of $\cE_i$.
If one can associate a value $x(\cE_i) \in (0,1)$ with each event $\cE_i$ such that
\begin{equation}
\EquationName{LLL}
\prob{\cE_i} ~\leq~ x(\cE_i) \cdot \prod_{\cF \in \Gamma(\cE_i)} \big(1-x(\cF)\big)
\end{equation}
then, with positive probability, no event $\cE_i$ occurs.
\end{theorem}

The weight that we assign to $\cE_{i,k}$ is 
\begin{equation}
\EquationName{weightdef}
    x(\cE_{i,k})
        ~:=~ 2 \exp\big( - \epsilon^2 \card{S_{i,k}}/16 - \epsilon \alpha 2^{k/2} /2 \big).
\end{equation}
Comparing to \eqref{eq:pikdef}, we see that this value is closely related to
(but slightly larger than) $p_{i,k}$,
which is our upper bound on the probability of $\cE_{i,k}$.

\begin{claim}
\ClaimName{xsmall}
$x(\cE_{i,k}) < 1/2$ for every $i \in [n]$ and $k \geq b$.
\end{claim}
\begin{proof}
By \eqref{eq:epsdef} we have $\epsilon > 4 \sqrt{\beta}$, so
$$
    \epsilon 2^{k/2}
        ~\geq~ \epsilon \sqrt{2^b}
        ~\geq~ \epsilon \sqrt{2^{\lg(1/\beta)-1}}
        ~\geq~ \epsilon \sqrt{1/2\beta}
        ~>~ 2 \sqrt{2}.
$$
It follows that $x(\cE_{i,k}) \leq 2 \exp( - \epsilon 2^{k/2}/2) < 2 \exp( - \sqrt{2} ) < 1/2$.
\end{proof}

Our next step is to characterize $\Gamma(\cE_{i,k})$,
the events that are dependent on $\cE_{i,k}$.
We let $\cC_{j,k}$ be the events corresponding to all entries 
of value roughly $2^{-k}$ in the $j\th$ column.
$$
    \cC_{j,k} ~\defeq~ \setst{ \cE_{i,k} }{ \floor{-\lg A_{i,j}} = k }
    \qquad(\text{for } j \in [m],~ k \geq b) 
$$
Next, $\cY_j$ contains all events corresponding to all entries
in the $j\th$ column. In other words, $\cY_j$ is the set of all events
that depend on the random variable $y_j$.
$$
\cY_{j}
    ~\defeq~ \bigcup_{k \geq b} \cC_{j,k}
    ~=~ \setst{ \cE_{i,\floor{-\lg A_{i,j}}} }{ i \in [n] }
    \qquad(\text{for } j \in [m])
$$
Finally, since $\cE_{i,k}$ depends only on 
the random labels of elements in $S_{i,k}$,
the set $\Gamma(\cE_{i,k})$ consists of all events that depend on any of those labels.
$$
\Gamma(\cE_{i,k}) ~=~ \Union_{j \in S_{i,k}} \cY_{j}.
$$

\begin{claim}
\ClaimName{LLLworks}
For every event $\cE_{i,k}$, inequality \eqref{eq:LLL} is satisfied.
\end{claim}
\begin{proof}
The main goal of the proof is to give a good lower bound for
$\prod_{\cF \in \Gamma(\cE_{i,k})} (1-x(\cF))$.
\Claim{xsmall} shows that $x(\cF) \leq 1/2$, so
\begin{equation}
\EquationName{prodtoexp}
\prod_{\cF \in \Gamma(\cE_{i,k})} (1-x(\cF))
    ~\geq~ \prod_{\cF \in \Gamma(\cE_{i,k})} \exp( -2 x(\cF))
    ~=~ \exp\Bigg( - 2 \sum_{\cF \in \Gamma(\cE_{i,k})} x(\cF) \Bigg).
\end{equation}
So it suffices to give a good upper bound for $\sum_{\cF \in \Gamma(\cE_{i,k})} x(\cF)$.

First we need to derive an inequality that is rather brutal, but suffices for our proof.
\begin{align}\nonumber
    \epsilon \cdot \alpha 2^{k/2} / 2
    &~=~ 8 \alpha \sqrt{\beta} \cdot \alpha 2^{k/2} /2 
        \qquad\text{(by \eqref{eq:epsdef})} \\\nonumber
    &~=~ \alpha^2 \cdot 2 \sqrt{\beta} \cdot 2^{1+b/2 + (k-b)/2} \\\nonumber
    &~=~ \lg(\delta/\beta^2) \cdot \big(2 \sqrt{\beta} 2^{b/2} \big) \cdot 2^{1+(k-b)/2} \\\nonumber
    &~\geq~ \big(b+\lg(\delta/\beta)\big) \cdot 2^{1+(k-b)/2}
        \qquad\text{(since $\lg(1/\beta) \geq b$ and $2^{b/2} \geq \sqrt{1/2\beta}$)}
        \\\nonumber
    &~\geq~ \big(b+\lg(\delta/\beta)\big) + 2^{1+(k-b)/2}
        \qquad\text{(since $xy \geq x+y$ if $x,y \geq 2$)}\\\nonumber
    &~\geq~ \big(b+\lg(\delta/\beta)\big) + (k-b)
        \qquad\text{(since $2^{1+i/2} \geq i$ for all $i \geq 0$)}
        \\\EquationName{eps2k}
    &~=~ k + \lg(\delta/\beta)
\end{align}
Next, consider all the events that depend on $y_j$.
Then
\begin{align*}
\sum_{\cF \in \cY_j} x(\cF)
    &~=~ \sum_{k \geq b} \sum_{\cF \in \cC_{j,k}} x(\cF) \\
    &~\leq~ \sum_{k \geq b} \sum_{\cF \in \cC_{j,k}} \exp( - \epsilon \alpha 2^{k/2} /2 )
        \qquad\text{(by \eqref{eq:weightdef})}\\
    &~\leq~ \sum_{k \geq b} \card{\cC_{j,k}} \cdot e^{-(k+\lg(\delta/\beta))}
        \qquad\text{(by \eqref{eq:eps2k})}\\
    &~\leq~ \sum_{k \geq b}
        \Big| \setst{i}{A_{i,j} \in (2^{-k-1},2^{-k}] } \Big| \cdot 2^{-(k+\lg(\delta/\beta))} \\
    &~\leq~ (\beta/\delta) \cdot (2\delta) ~=~ 2 \beta,
\end{align*}
since the $j\th$ column sums to $\delta$.
Therefore
$$
    \sum_{\cF \in \Gamma(\cE_{i,k})} x(\cF)
        ~=~ \sum_{j \in S_{i,k}} \sum_{\cF \in \cY_j} x(\cF)
        ~\leq~ 2 \card{S_{i,k}} \beta.
$$
Combining this with \eqref{eq:prodtoexp}, we obtain the lower bound
\begin{align*}
x(\cE_{i,k}) \cdot \prod_{\cF \in \Gamma(\cE_{i,k})} (1-x(\cF))
    &~\geq~ x(\cE_{i,k}) \cdot \exp\Bigg( - 2 \sum_{\cF \in \Gamma(\cE_{i,k})} x(\cF) \Bigg) \\
    &~\geq~ 2 \exp\big( - \epsilon^2 \card{S_{i,k}} /16 - \epsilon \alpha 2^{k/2} /2 \big)
        \cdot \exp\big( - 4 \card{S_{i,k}} \beta \big) \\
    &~=~ 2 \exp\Big( - \card{S_{i,k}} (\epsilon^2 /16 + 4 \beta)
        - \epsilon \alpha 2^{k/2} /2 \Big) \\
    &~\geq~ 2 \exp\Big( - \card{S_{i,k}} \epsilon^2/8
        - \epsilon \alpha 2^{k/2} /2 \Big) \\
    &~=~ p_{i,k} ~\geq~ \prob{\cE_{i,k}}
\end{align*}
where the penultimate inequality holds because $\epsilon^2/8 \geq \epsilon^2/16 + 4 \beta$,
which follows because $\epsilon \geq 8 \sqrt{\beta}$ (cf.~\eqref{eq:epsdef}).
This proves \eqref{eq:LLL}.
\end{proof}

The previous claim shows that the hypotheses of the local lemma are satisfied.
So there exists a vector $y \in \set{-1,+1}^m$
such that none of the events $\cE_{i,k}$ hold.
As in \eqref{eq:nobad}, this implies that every row has discrepancy
at most $16 \alpha \sqrt{\beta}$.
In other words, $\norm{A y}_\infty \leq 16 \alpha \sqrt{\beta}$.
This completes the proof of \Theorem{weakbeckfiala}.

\section{Conclusion}

Many discrepancy theorems on hypergraphs have a 
more general statement about the discrepancy of real-valued matrices \cite[\S 4]{BeckSos}.
We have provided another occurrence of this phenomenon
by proving \Theorem{newthm}, which generalizes \Theorem{folklore}.

We are not aware of any result showing that
either Theorem~\ref{thm:folklore} or \ref{thm:newthm}
is optimal. It seems conceivable that the logarithmic factor could be removed.

\begin{conj}
Let $V$ be an $n \times m$ real matrix with $\abs{V_{i,j}} \leq 1$,
$\norm{ v^i }_1 \leq R$, and $\norm{v_j}_1 \leq \Delta$ for all $i \in [n], j \in [m]$.
Assume $R \geq \Delta$.
There exists $y \in \set{-1,+1}^m$ with $\norm{ V y }_\infty \leq O(\sqrt{R})$.
\end{conj}

Let us mention now the recent discrepancy result of Marcus et al.~\cite{MSS},
which implies a solution to the long-standing Kadison-Singer problem.

\begin{theorem}[Corollary 1.3 of Marcus et al.~\cite{MSS}]
\TheoremName{MSS}
Let $u_1,\ldots,u_m \in \bC^n$ satisfy $\sum_{i=1}^m u_i u_i^* = I$
and $\norm{u_i}_2^2 \leq \delta$ for all $i$.
Then there exists $y \in \set{-1,+1}^m$ such that
$
\norm{ \sum_{i=1}^m y_i u_i u_i^* } ~\leq~ O(\sqrt{\delta}),
$
where $\norm{\cdot}$ is the $\ell_2$-operator norm.
\end{theorem}

There is a relationship between 
Theorems~\ref{thm:weakbeckfiala} and \ref{thm:MSS},
in the sense that both are implied by the following conjecture.
\Theorem{MSS} is the special case where each $A_i$ has rank one,
and \Theorem{weakbeckfiala} implies (ignoring the additional logarithmic factor $\alpha$)
the special case where each $A_i$ is a diagonal matrix.

\begin{conj}
Let $A_1,\ldots,A_m$ be Hermitian, positive semi-definite matrices of the same size
satisfying $\sum_{i=1}^m A_i = I$ and $\operatorname{tr} A_i \leq \delta$ for all $i$.
There exists $y \in \set{-1,+1}^m$ with $\norm{ \sum_{i=1}^m y_i A_i } \leq O(\sqrt{\delta})$.
\end{conj}



\end{document}